\documentclass[oneside,11pt]{article}

\usepackage{amsmath}
\usepackage{amssymb}
\usepackage{pxfonts}
\usepackage{dsfont}
\usepackage{graphicx}
\usepackage{eucal}
\usepackage{mathrsfs}
\usepackage{theorem}
\usepackage{pifont}
\usepackage{sectsty}
\usepackage{amscd}
\usepackage{color}
\usepackage{fancyhdr}
\usepackage{framed}
\usepackage[all]{xy}
\usepackage{hyperref}
\usepackage[shortlabels]{enumitem}

\definecolor{shadecolor}{rgb}{0.8,0.8,0.8}

\theoremheaderfont{\fontfamily{pzc}\bfseries\large}
\newtheorem{theorem}{Theorem}[section]

\newtheorem{proposition}[theorem]{Proposition}

\newtheorem{corollary}[theorem]{Corollary}

{\theorembodyfont{\rmfamily}

\newcommand{\specexercise}[1]{}

\newenvironment{proof}{{\flushleft \emph{Proof}:}}{\hfill\ding{110}}

\newcommand{\Vol}{\text{Vol}}

\newcommand{\R}{\mathbb{R}}

\renewcommand{\det}{\operatorname{det}}

\newcommand{\W}{\Omega}

\newcommand{\weakly}[1]{\stackrel{#1}{\rightharpoonup}}

\newcommand{\brk}[1]{\left(#1\right)}          
\newcommand{\BRK}[1]{\left\{#1\right\}}        

\numberwithin{equation}{section}

\begin{document}

\title{A simple example of the weak discontinuity of $f\mapsto \int \det \nabla f$}
\author{Cy Maor\footnote{Einstein Institute of Mathematics, The Hebrew University, Jerusalem, Israel.}}
\date{}
\maketitle

\vspace{-0.5cm}
\begin{abstract}
Verifying lower-semicontinuity of integral functionals in the weak topology of Sobolev spaces is a central theme in the calculus of variations.
For integral functionals with $p$-growth, quasiconvexity is a necessary condition for weak lower-semicontinuity in $W^{1,p}$, but is only sufficient if some additional conditions are met.
The standard functional showing the necessity of additional conditions is $f\mapsto \int_\W \det \nabla f$, which fails to be weakly lower-semicontinuous.
However, the common examples showing this failure are non-injective and have a lot of shear.
The aim of this short note is to point out that a known sequence of conformal diffeomorphisms of the $d$-dimensional unit ball that converges weakly to a constant in $W^{1,d}$, exemplifies the weak discontinuity of this functional even when restricting a space to functions which are ``as nice as possible''.
\end{abstract}

\section{Introduction}
Let $\W\subset \R^d$ be a Lipschitz domain, and let $p\in(1,\infty)$.
Let $W:\R^{d\times D} \to \R$ be a continuous function satisfying $p$-growth, i.e.,
\[
|W(A)| \le \alpha(1+|A|^p)
\]
for some $\alpha\ge 0$, and consider the functional
\[
I:W^{1,p}(\W;\R^D) \to \R,\qquad I(f) = \int_\W W(\nabla f(x))\,dx.
\]
One of the fundamental questions in the calculus of variations is whether such functionals are lower-semicontinuous with respect the weak topology of $W^{1,p}$.
The fundamental result is that for weak lower-semicontinuity, $W$ must be \emph{quasiconvex}, i.e., ``convex over gradients'',
\[
W(A) \le \int_{(0,1)^d} W(A + \nabla \phi(x))\,dx , \qquad \phi\in W^{1,\infty}_0((0,1)^d;\R^D).
\]
Moreover, quasiconvexity is sufficient given some additional information is given, e.g.,
\begin{enumerate}[(1)]
\item Restricting $W$: Assuming that $\lim_{|A|\to \infty} \frac{W^-(A)}{1+|A|^p} = 0$, where $W^- = \max\{-W,0\}$.
\item Restricting the space of functions: Considering only functions satisfying given Dirichlet boundary data.
\end{enumerate}
See \cite[Definition 3.10]{BK17} for more details, also \cite[Theorem 8.4]{Dac08} and \cite{KK08}.

The standard functional for showing that some form of additional information is needed, is the $d$-growth functional 
\[
I_{\det} : W^{1,d}(\W;\R^d)\to \R, \qquad I_{\det}(f) = \int_\W \det \nabla f(x)\,dx,
\]
which is neither lower-, nor upper-semicontinuous in the weak $W^{1,d}$-topology, despite the function $A\mapsto \det A$ being quasiconvex (and actually quasiaffine).

The classical example for this failure of weak semicontinuity, which is attributed to Tartar, is given in \cite[Example 7.3]{BM84} (also in \cite[Example 8.6]{Dac08}, \cite[Example 3.6]{BK17}):
\[
f_n:(0,a)^2 \to \R^2, \qquad f_n(x,y) := \frac{1}{\sqrt{n}} (1-y)^n(\sin nx,\cos nx),
\]
for $a\in (0,1)$.
A simple calculation shows that $f_n\to 0$ uniformly and weakly in $W^{1,2}$, however,
\[
\lim_{n\to\infty} \int_{(0,a)^2} \det \nabla f_n = - \frac{a}{2} < 0.
\]
Note that the functions $f_n$ are highly oscillatory and non-injective, have a lot of shear (i.e., they change angles radically), and their images converge to zero (since $f_n\to 0$ uniformly).
Another example appears in \cite[Example 3.6]{BK17}; this one has a constant image, but again, it is non-injective (the support of the sequence vanishes asymptotically).

These examples raise the question whether by restricting the space of functions to embeddings, one can recover weak semicontinuity: 
\begin{quote}
Does $\int_\W \det \nabla f_n(x)\,dx \to 0$ for a sequence of \emph{embeddings} $f_n:\W\to \R^d$ which converges weakly in $W^{1,d}$ to zero?
\end{quote}
In this short note we show that the answer to this question is negative; the example is given by very simple conformal mappings of the unit ball, whose images are a fixed ball, yet $f_n\weakly{} 0$. 
This example already appeared in the other contexts, see e.g., \cite[Example~2.1]{IO11} (also \cite[Remark~2.3]{IOPR21}).
The failure of continuity is due to concentration of $\|\nabla f_n\|_{L^d}$ near the boundary. This is a general phenomenon, as discussed in \cite[Section 3]{BK17}.

This example should also be compared with \cite{Mul90}, which shows that if $f_n \weakly{} 0$ are embeddings (or more generally, maps whose Jacobians have a sign) then $\det \nabla f_n \weakly{} 0$ \emph{locally} in $L^1$ (this result is part of the large literature on convergence of $\det \nabla f_n$ in various topologies; see, e.g., \cite{FLM05} and the references therein). 
The conformal example shows that a (global) weak $L^1$ convergence can fail for functions ``as nice as possible".
\section{The example}
\begin{proposition}[Counterexample for weak lower-semicontinuity]
The M\"obius transformations 
\begin{equation}
\label{eq:example}
f_n:B_1(0)\to \R^d, \qquad f_n(x) = b_n +\frac{2n+1}{n^2} \frac{x-a_n}{|x-a_n|^2},
\end{equation}
where
\[
a_n = (1+1/n,0,\ldots,0), \qquad b_n = (1/n,0,\ldots,0),
\]
are diffeomorphism of the unit ball $B_1(0)$ into a fixed unit ball, that satisfy
\[
f_n \weakly{} 0\,\,\, \text{in} \,\, W^{1,d}(B_1(0),\R^d),\quad \text{and} \quad \int_{B_1(0)} \det \nabla f_n = -\Vol(B_1(0)).
\]
\end{proposition}

This sequence of diffeomorphisms thus exemplifies the failure of $I_{\det}$ to be weakly lower-semicontinuous in $W^{1,d}$.
By translating $f_n$ by $(1,0,\ldots,0)$, one can obtain a sequence of diffeomorphism of $B_1(0)$ into itself, that converge to a point.
Note that it is a sequence of orientation-reversing maps (i.e., maps of negative Jacobians); this is essential, since restricting to maps of non-negative Jacobians, $I_{\det}$ is weakly lower-semicontinuous, as for such maps
\[
I_{\det}(f) = \int_\W \max\{\det \nabla f(x) , 0\}\,dx, 
\]
and the function $W(A) = \max\{\det A,0\}$ is a quasiconvex function satisfying condition (1) above.

\begin{proof}
$f_n$ are M\"obius transformations, and hence map injectively spheres and hyperplanes into spheres an hyperplanes.
Specifically, since $a_n$ is in the direction of $e_1=(1,0,\ldots,0)$ they map the unit sphere into a sphere for which $f(\pm e_1)$ are antipodal points.
Since 
\[
f_n(1,0,\ldots,0) = \brk{-2,0,\ldots,0}, \quad f_n(-1,0,\ldots,0) =0,
\]
the image of $B_1(0)$ is the ball $B_1(-e_1)$.
Since $f_n$ is an inversion of a sphere, it has a negative Jacobian, and thus, since $f_n$ is injective, we obtain that $\int_{B_1(0)} \det \nabla f_n =- \Vol(f_n(B_1(0))) = - \Vol(B_1(0))$.

Since $f_n$ are M\"obius transformations with negative Jacobians, they are anticonformal maps, hence $\nabla f_n^T\nabla f_n = |\det \nabla f_n|^{2/d} I$, and so $\|\nabla f_n\|_{L^d(B_1(0);\R^{d\times d})}$ does not depend on $n$. 
Since $f_n(B_1(0))=B_1(-e_1)$, it follows that $\| f_n\|_{W^{1,d}(B_1(0);\R^d)}$ is uniformly bounded.
For any $x\in B_1(0)$ we have that $|x-a_n|\nrightarrow 0$, hence $f_n(x) \to 0$ pointwise.
Thus $f_n \weakly{} 0$ in $W^{1,d}(B_1(0);\R^d)$.
\end{proof}

\vspace{0.4cm}
We conclude with another simple and nice corollary that follows from this example:
\begin{corollary}
For a given $c>0$, let
\[
\mathcal{A}_c = \BRK{f\in W^{1,d}(B_1(0);\R^d) : \int_{B_1(0)} \det \nabla f = c}, \quad \mathcal{B}_c = \BRK{f\in W^{1,d}(B_1(0);\R^d) : \Vol(f(B_1(0)) = c} .
\]
Then
\begin{equation}
\label{eq:inf_energy}
\inf_{\mathcal{A}_c} \|f\|_{W^{1,d}}^d = \inf_{\mathcal{B}_c} \|f\|_{W^{1,d}}^d = d^{d/2} c,
\end{equation}
and the infimum is not attained.
\end{corollary}

\begin{proof}
Recall the for a matrix $A\in \R^{d\times d}$ with singular values $\sigma_1,\ldots,\sigma_d$, we have that $|A|^2 = \sum_1^d \sigma_i^2$ and $|\det A| = \Pi_1^d \sigma_i$. 
Thus, by the arithmetic-geometric mean inequality, $|A|^d \ge d^{d/2} |\det A|$, with equality holding if and only if all the singular values are the same, i.e., if $A$ is conformal or anticonformal.
Applying this to a map $f\in W^{1,d}(B_1(0);\R^d)$, we obtain
\begin{equation}
\label{eq:inequalities}
\|f\|_{W^{1,d}}^d \ge  \int_{B_1(0)} |\nabla f(x)|^d \, dx \ge  d^{d/2}\int_{B_1(0)} |\det \nabla f (x)| \,dx,
\end{equation}
where equality in the second inequality holds if and only if $f$ is conformal or anticonformal.
The righthand side is larger or equal than both $\int_{B_1(0)} \det \nabla f$ and $\Vol(f(B_1(0))$, and therefore 
\[
\inf_{\mathcal{B}_c} \|f\|_{W^{1,d}}^d \ge d^{d/2} c,\quad  \inf_{\mathcal{A}_c} \|f\|_{W^{1,d}}^d \ge d^{d/2} c.
\]
Let $g:\R^d\to \R^d$ be the reflection of the $e_1$ axis, and consider the functions $g\circ f_n$, where $f_n$ are as in \eqref{eq:example} (the composition with $g$ is just to obtain positive Jacobians).
These are conformal maps, hence $|\nabla (g\circ f_n)|^d = d^{d/2} \det (\nabla (g\circ f_n))$, mapping $B_1(0)$ into $B_1(e_1)$, and converging weakly to zero in $W^{1,d}$.
Thus, by letting $c' = (c/\Vol (B_1(0)))^{1/d}$ and defining $\tilde{f}_n := c' g\circ f_n$ we obtain that $\tilde{f}_n  \in\mathcal{A}_c\cap \mathcal{B}_c$, and
\[
\lim_{n\to \infty} \|\nabla \tilde{f}_n\|_{W^{1,d}}^d = d^{d/2}c,
\]
since $\tilde{f}_n$ are conformal and $\tilde{f}_n\to 0$ strongly in $L^d$. 
This proves \eqref{eq:inf_energy}.
If the infimum is attained by some $f$ then, from \eqref{eq:inequalities},
\[
\|f\|_{W^{1,d}}^d =  \int_{B_1(0)} |\nabla f(x)|^d \, dx,
\]
which implies $f=0$, in contradiction.
\end{proof}

{\footnotesize
\bibliographystyle{amsalpha}
\bibliography{/Users/cy/MyDrive/MyBib.bib}			
}

\end{document}